\documentclass[11pt,a4paper,twoside]{article}
\usepackage[utf8]{inputenc}
\usepackage{amsmath,amsfonts,amssymb,amsthm,bbm,latexsym,mathrsfs}
\usepackage{graphicx,color,epsfig,fancyhdr,dsfont}
\usepackage{enumerate}
\usepackage{hyperref}
\usepackage{indentfirst}
\usepackage[all]{hypcap}
\usepackage{placeins}
\usepackage[affil-it]{authblk}
\usepackage{color}
\usepackage[includeheadfoot,margin=3.5cm]{geometry}
\usepackage{subcaption}
\usepackage{xcolor}
\usepackage{float}
\allowdisplaybreaks[4]

\allowdisplaybreaks

\newtheorem{theorem}{Theorem}
\newtheorem{definition}{Definition}
\newtheorem{lemma}{Lemma}
\newtheorem{assumption}{Assumption}

\newtheorem{remark}{Remark}

\def\b1{{\mathbbm{1}}}

\numberwithin{equation}{section}

\newcommand{{\X}}{{\mathbb{X}}}                     

\title{Ergodic Stochastic Optimal Control Problems}

\author[1]{Chenglin Ma}
\author[2,3]{Huaizhong Zhao}
\affil[1]{Research Center for Mathematics and Interdisciplinary Sciences, Shandong University, Qingdao 266237, China}
\affil[2]{Department of Mathematical Sciences, Durham University, DH1 3LE, UK}
\affil[3]{School of Mathematics, Shandong University, Jinan 250100, China}
\affil[ ]{machenglin@sdu.edu.cn, huaizhong.zhao@durham.ac.uk}
\date{~}
\begin{document}
\maketitle
\begin{abstract}
In this article, we introduce a novel approach to solving the ergodic stochastic optimal control problem whose dynamics is driven by a controlled stochastic differential equation. The coefficients of this stochastic differential equation are non-autonomous but periodic in time. We first prove that the infinite horizon average stochastic optimal control problem is ergodic, \emph{i.e.}, the value function equals a constant $\rho$ which is independent of the initial conditions. Based on the result of ergodicity, we construct an auxiliary function $w(t,x)$ that is well-defined and periodic in time. We can prove that the pair $(w,\rho)$ satisfies the dynamic programming principle and is a viscosity solution of the associated Hamilton-Jacobi-Bellman equation. Finally, we apply our results to the study of ergodic backward stochastic differential equations. Our method is even new in the homogeneous case.
\end{abstract}

\section{Introduction}
The ergodic stochastic optimal control problem is one of the most important criteria in control theory, the goal of which is to maximize the ergodic payoff functional over $\mathcal{U}$ the set of all admissible controls, \emph{i.e.}, the value function is given as
\begin{align*}
\rho(x)=\sup_{u\in\mathcal{U}}\limsup_{T\to\infty}\frac{1}{T}E\int_0^Tf(s,X_s,u_s)ds,
\end{align*}
subject to an infinite horizon stochastic controlled system
\begin{align*}
\begin{cases}
dX_s=b(s,X_s,u_s)ds+\sigma(s,X_s,u_s)dB_s,~~s\geq 0,\\
X_0=x\in\mathbb{R}^d.
\end{cases}
\end{align*}
Here, all the coefficients are periodic with a common period $\tau>0$.

Ergodicity is significant in the study of random dynamical systems in describing their large-time behaviour. Ergodicity for Markov chains has attracted attention from many researchers,  e.g. Down-Meyn-Tweedie \cite{10.1214/aop/1176987798}, Meyn-Tweedie \cite{Meyn2009}, and Miller \cite{Miller1966}, to name but a few. Harris \cite{Har56} introduced the method of recurrence and small set to study the ergodicity of Markov chains, 
and later Hairer-Mattingly \cite{hairer2011} refined the Lyapunov condition and the local Doeblin condition to give an explicit formula of the contraction rate in the stationary case.

Random periodic processes are ubiquitous in the real world, e.g. appearing in economics, climate dynamics, etc. They have been studied, together with an ergodic theory, in the past 15 years (see Feng-Zhao \cite{FZ2020} and references therein). 
The study of stochastic controls associated with these problems is an important problem in applications, but a general mathematical theory is almost empty. As far as we know, there is only one work on periodic controlled stochastic systems (Sun-Yong \cite{SUN2024189}), which is for a linear-quadratic control problem. 
This is very different from the problem in a stationary regime. For example, an economic cycle contains the expansion and contraction stages, at each control strategies should be different to reflect the random periodicity.

A basic framework solving an stochastic optimal control problem is to prove that the value function satisfies the dynamic programming principle (DPP) based on the pioneering work of Bellman \cite{B}, and is the unique viscosity solution or Sobolev weak solution of the associated Hamilton-Jacobi-Bellman (HJB) equation, see, e.g. Buckdahn-Ma \cite{doi:10.1137/S036301290444335X}, Ma-Zhao \cite{10130640}, Pham-Wei \cite{doi:10.1137/16M1071390}, Wei-Wu-Zhao \cite{Wu2014} and Yong-Zhou \cite{YZ}. However, this method may fail to work in infinite-horizon ergodic stochastic optimal control problems, since it is not easy to find the dynamic programming equation satisfied by $\rho(x)$. Finding a way to solve the ergodic control problems becomes an interesting topic.

Generally speaking, there have been two approaches to studying the ergodic optimal control problem. One is to directly consider the ergodic HJB equation 
\begin{align*}
\rho=\sup_{u\in U}\{f(x,u)+Dw(x)\cdot b(x,u)+\frac{1}{2}tr[\sigma\sigma'(x,u)D^2w(x)]\},
\end{align*}
whose viscosity solution pair $(w,\rho)$ is given by the asymptotic limit as $\alpha\downarrow0$ of the corresponding infinite-horizon $\alpha$-discounted HJB equation, e.g. Arisawa-Lions \cite{AL1998}, Bao-Tang \cite{BAO2023127404}, Cosso-Fuhrman-Pham \cite{COSSO20161932} and Frehse-Bensoussan \cite{FrehseBensoussan}. Compared with the traditional DPP method, this approach is in some sense a reverse process.

Another approach is to consider an ergodic forward-backward stochastic differential equation (FBSDE) introduced in Fuhrman-Hu-Tessitore \cite{HuYing2009},
and the Girsanov transformation was used to introduce a weak formulation of a control through a family of equivalent probability measures. Thus, the following ergodic FBSDE system with a control was subsequently studied 
\begin{align*}
  \begin{cases}
    dX_s^x=(AX_s^x+F(X_s^x))dt+G(dW_s+R(u_s)ds),~~X_0=x,\\
    dY_s^x=-[\psi(X_s^x,Z_s^x)-Z_sR(u_s)-\rho]ds+Z_sdW_s.
  \end{cases}
\end{align*} 
It was proved that 
\begin{align*}
\rho=\sup_{u}\limsup_{T\to\infty}\frac{1}{T}E\int_0^Tf(X_s^x,u_s)ds,
\end{align*}
if $\psi(x,z)=\sup_{u\in U}\{f(x,u)+zR(u)\}$, see also Guatteri-Tessitore \cite{Giuseppina2020}, Jackson-Liang \cite{Jackson-Liang} and Hu-Liang-Tang \cite{HLT2020}. In order to solve the ergodic FBSDE problem, 
a global dissipative condition was usually imposed on $Ax+F(x)$.

In this article, we investigate an ergodic stochastic control problem with periodic coefficients under a weakly dissipative condition. The key innovation is that we introduce an entirely new method. 
The specific novelties and contributions of the current paper can be summarized as follows.

(1) Unlike existing approaches to studying ergodic stochastic control problems, all of the results in this paper do not depend on the study of the associated discounted control problems. We first prove that there exists a constant $\rho$ that is independent of the initial conditions such that 
\begin{align*}
  \rho=\sup_u\lim_{T\to\infty}\frac{1}{T}E\int_0^Tf(s,X_s,u(s,X_s))ds.
\end{align*}
With this result, we construct an auxiliary function $(t,x)\mapsto w(t,x)$, and verify that the pair $(w,\rho)$ satisfies the dynamic programming principle and is a viscosity solution of the related HJB equation: 
\begin{align*}
\partial_s\varphi(t,x)+\sup\limits_{u\in U}[f(t,x,u)+D\varphi(t,x)\cdot b(t,x,u)+\frac{1}{2}tr(\sigma\sigma'(t,x,u)D^2\varphi(t,x))]=\rho,\\
(t,x)\in[0,\tau)\times\mathbb{R}^d,
\end{align*}
with $\varphi(0,\cdot)=\varphi(\tau,\cdot)$.

(2) We assume that the drift coefficient is weakly dissipative (Assumption \ref{assume}), and the controlled diffusion coefficient is non-degenerate. To the best of our knowledge, weakly dissipativity has received little attention in the existing literature on ergodic control and ergodic BSDE. Existing approaches typically rely on a global dissipative condition, which ensures that trajectories starting from different initial states contract pathwise and provides the key ingredient for establishing long-time convergence. However, 
we show that the convergence only holds in the sense of law under a weakly dissipative condition. For this, we give a completely different method in this paper. 
This probabilistic framework replaces trajectory-wise contraction and provides the foundation for proving ergodicity for the average control problem, constructing the auxiliary periodic value function, and deriving the associated dynamic programming principle and HJB equation. 
It is worth pointing out that a special weakly dissipative condition
   $\langle Ax,x\rangle\leq -K|x|^2, |F(x)|\leq M$,
for some constant $K,M>0$, was imposed in Debussche-Hu-Tessitore \cite{DEBUSSCHE2011407} and Madec \cite{Madec2015}.

(3) We give the ergodicity of the controlled Markov process under the periodic setting, and especially, we give an infinite-horizon periodic HJB equation. To the best of our knowledge, this is the first to give the infinite-horizon time-dependent HJB equation (\ref{PDE}). It is worth pointing out that the ergodic HJB equation is very different from the HJB equation of finite time stochastic control problems which has a terminal condition, 
(\ref{PDE}) can be seen that the terminal condition is forgotten during the process of the ergodicity taking terminal time to infinity.

This paper provides a new methodology for ergodic stochastic control problems. Even for homogeneous stochastic controlled systems, this method is still novel.

The remainder of this article is organized as follows. In the following section, we mainly introduce the setting of this problem. In Section \ref{convergency results}, we prove the ergodicity results. In Section \ref{dynamic programming principle}, we construct the auxiliary function based on the ergodicity results and give the dynamic programming equation satisfied by this auxiliary function. In Section \ref{HJB}, we give the related HJB equation and prove that the pair $(w,\rho)$ is its viscosity solution.

\section{Preliminaries}\label{control}
Consider a type of stochastic optimal control problem whose dynamics is driven by an infinite-horizon $d$-dimensional controlled SDE
\begin{equation}\label{dynamics}
\begin{cases}
dX_s^{t,x,u}=b(s,X_s^{t,x,u},u_s)ds+\sigma(s,X_s^{t,x,u},u_s)dB_s,\ s\geq t,\\
X_t^{t,x,u}=x\in\mathbb{R}^d, 
\end{cases}
\end{equation}
where $\{B_s,s\geq t\}$ is a $d$-dimensional Brownian motion defined on the completed probability space $(\Omega,\mathcal{F},P)$ and the control process $u(\cdot)$ is a feedback one taking values in a nonempty compact set $U\subset\mathbb{R}^k$, \emph{i.e.}, $u_s(\omega)=u(s,X_s(\omega))\in U$, where $u$ is a measurable deterministic function defined on $[0,\infty)\times\mathbb{R}^d$. 
We denote by $\mathcal{A}$ the set of all functions $u(s,x)$ that are c\`adl\`ag in $s$ uniformly in $x$ and Lipschitz continuous in $x$ uniformly in $s$, for example, $u(s,x)=u_1(s)+u_2(x)$ with c\`adl\`ag function $u_1$ and Lipschitz continuous function $u_2$.

The average payoff functional is given by
\begin{equation}\label{payoff functional}
J(t,T,x,u(\cdot))=\frac{1}{T}E\left[\int_t^{t+T}f(s,X_s^{t,x,u},u(s,X_s^{t,x,u}))ds\right].
\end{equation} 
Throughout this article, the mappings 
\begin{align*}
b,\sigma,f:[0,\infty)\times\mathbb{R}^d\times U\to\mathbb{R}^d,\mathbb{R}^{d\times d},\mathbb{R},
\end{align*}
are always assumed to satisfy the following conditions.

\begin{assumption}\label{assume}
For $l\in\{b,\sigma,f\}$.\\
1)\ The mapping $l(\cdot,x,\cdot)$ is continuous uniformly in $x$.\\
2)\ There exists some constant $C_l>0$ such that for all $s\geq0$, $x,y\in\mathbb{R}^d$ and $u\in\mathcal{A}$ 
\begin{align*}
|l(s,x,u(s,x))-l(s,y,u(s,y))|\leq C_l|x-y|.
\end{align*}
3)\ There exist some constants $L',\lambda>0$ such that for all $s\geq0$, $x\in\mathbb{R}^d$ and $v\in U$
\begin{align*}
2\langle x,b(s,x,v)\rangle\leq L'-\lambda|x|^2.
\end{align*}
4)\ There exist $0<\underline{\sigma}<\overline\sigma$ such that for all $s\geq0$, $x\in\mathbb{R}^d$ and $v\in U$
\begin{align*}
  \underline\sigma |\xi|^2\leq\langle\sigma\sigma'(s,x,v)\xi,\xi\rangle\leq\overline\sigma|\xi|^2.
\end{align*}
5)\ There exists a minimal $\tau>0$ such that for any $s\geq0$
\begin{align*}
l(s,\cdot,\cdot)=l(s+\tau,\cdot,\cdot).
\end{align*}
\end{assumption}

\begin{remark}
The Assumption \ref{assume} implies the global linear growth condition of the mappings $b, \sigma$ and $f$, i.e., there exists some constant $C>0$ such that for all $s\geq0$, $x\in\mathbb{R}^d$ and $u\in \mathcal{A}$
\begin{align*}
  |b(s,x,u(s,x))|+\|\sigma(s,x,u(s,x))\|+|f(s,x,u(s,x))|\leq C(1+|x|).
\end{align*}
It also implies that there exist some constants $L,\lambda>0$ such that 
\begin{align*}
2\langle x,b(s,x,u(s,x))\rangle+\|\sigma(s,x,u(s,x))\|^2\leq L-\lambda|x|^2.
\end{align*}
\end{remark}

For any $n\tau\leq t<(n+1)\tau$, we denote by $\mathcal{U}_t $ the set of all admissible feedback controls $u(\cdot)$ whose corresponding deterministic measurable functions belong to $\mathcal{A}$ and are $\tau$-periodic in time after $(n+1)\tau$, and $\mathcal{A}_t\subset\mathcal{U}_t$ the set of all admissible controls $u(\cdot)$ whose corresponding functions are $\tau$-periodic after $t$.

For any given $u(\cdot)\in\mathcal{U}_t$, it is well-known from Assumption \ref{assume} that the controlled SDE \eqref{dynamics} admits a unique strong solution. Moreover, for all $s\geq t$, $x,y\in\mathbb{R}^d$ and $u(\cdot)\in\mathcal{U}_t$
\begin{equation}\label{estimates}
\begin{aligned}
  &E|X_s^{t,x,u}|^2\leq \frac{L}{\lambda}+e^{-\lambda(s-t)}|x|^2,\\
  &E|X_s^{t,x,u}-X_s^{t,y,u}|^2\leq e^{(C_\sigma^2+2C_b)(s-t)}|x-y|^2.
\end{aligned}
\end{equation}
The value function associated with the average payoff functional (\ref{payoff functional}) is defined as
\begin{align}\label{value function}
\rho(t,x)=\sup_{u(\cdot)\in\mathcal{U}_t}\limsup_{T\to\infty}J(t,T,x,u(\cdot)).
\end{align}

If there exists a constant $\rho$ such that for all $t\geq 0$ and $x\in\mathbb{R}^d$,
\begin{align*}
\rho(t,x)\equiv\rho,
\end{align*}
we say that the long-time average stochastic optimal control problem is ergodic. Our main aim in this article is to study the ergodicity of the average control problem and find the optimal controls.

\section{Ergodicity results}\label{convergency results} 
In this section, we will investigate the ergodicity of the average stochastic optimal control problem as $T\to\infty$ under the periodic setting.
For any $t\geq 0$ and $u(\cdot)\in\mathcal{U}_t$, we first from (\ref{estimates}) have the Lyapunov condition: there exist constants $\gamma\in(0,1)$ and $K>0$ such that
\begin{align*}
E\left|X_{t+\tau}^{t,x,u}\right|\leq \gamma|x|+K,
\end{align*}
and the local Doeblin condition: there exist a probability measure $\nu(\cdot)$ and a constant $\alpha\in(0,1)$ such that 
\begin{align*}
\inf_{x\in\mathcal{C}}P_u(t,t+\tau,x,\cdot)\geq \alpha\nu(\cdot),
\end{align*}
where $P_u(t,t+\tau,x,\Gamma)=P\{X_{t+\tau}^{t,x,u}\in\Gamma\}$ and $\mathcal{C}=\{x:|x|\leq R\}$ for some $R>2K/(1-\gamma)$. For the derivation of the local Doeblin condition, see \cite{FENG202367} and \cite{MENOZZI2021330} which implies the uniformity of $\alpha$ and $\nu(\cdot)$ in $u(\cdot)$.

Define a metric on $\mathbb{R}^d$ by
\begin{equation*}
 d_\beta(x,y)=\begin{cases}
0,&x=y,\\
2+\beta|x|+\beta|y|,&x\neq y,
  \end{cases}
\end{equation*}
and a seminorm
\begin{align*}
   \|\varphi\|_\beta=\sup_{x\neq y}\frac{|\varphi(x)-\varphi(y)|}{d_\beta(x,y)}.
\end{align*}
Denote by $C_{lip}(\mathbb{R}^d)$ the space of all Lipschitz continuous functions on $\mathbb{R}^d$. 
For any fixed $u(\cdot)\in\mathcal{U}_t$ and $s\geq t$, define an operator $T_s^{t,u(\cdot)}:C_{lip}(\mathbb{R}^d)\to C_{lip}(\mathbb{R}^d)$ by
\begin{align*} 
\varphi\mapsto T_s^{t,u(\cdot)}[\varphi](x)=E[\varphi(X_s^{t,x,u})],
\end{align*}
which satisfies the Chapman formula: for any $t\leq r\leq s$ and $\varphi\in C_{lip}(\mathbb{R}^d)$
\begin{align*}
T_s^{t,u(\cdot)}[\varphi](x)=T^{t,u(\cdot)}_r\circ T^{r,u(\cdot)}_s[\varphi](x).
\end{align*}

The contraction result under the Lyapunov condition and the local Doeblin condition obtained in \cite{hairer2011} also works for the control problems.
\begin{theorem}\label{yasuo}
Under Assumption \ref{assume}, there exist some constants $\bar{\alpha}\in(0,1)$ and $\beta>0$ such that for any $\varphi\in C_{lip}(\mathbb{R}^d)$
\begin{align*}
   \left\|T^{t,u(\cdot)}_{t+\tau}[\varphi]\right\|_\beta\leq\bar\alpha\|\varphi\|_\beta.
\end{align*}
Actually,  we can let $R$ be large enough such that $2K/R<1-\gamma$ and $\beta$ be small enough such that $\beta K<\alpha$, and set $\gamma_0=\gamma+2K/R<1$, $\alpha_0=\beta K$, then $\bar{\alpha}=(1-\alpha+\alpha_0)\vee (2+R\beta\gamma_0)/(2+R\beta)$.
\end{theorem}

Before deducing the ergodicity results, we first give an important lemma.
\begin{lemma}\label{lemma 2.5}
    For any $u(\cdot)\in\mathcal{A}_t$, we have $T_s^{t,u(\cdot)}[\varphi](x)=T_{s+\tau}^{t+\tau,u(\cdot)}[\varphi](x)$, i.e.
    \begin{align*}
        E[\varphi(X_s^{t,x,u})]=E[\varphi(X_{s+\tau}^{t+\tau,x,u})].
    \end{align*}
\end{lemma}
\begin{proof}
Define a metric dynamical system $\theta:\mathbb{R}\times\Omega\to\Omega$ as $\theta_s\omega_\cdot=\omega_{s+\cdot}-\omega_s$, it satisfies $\theta_s P=P$ for all $s\in\mathbb{R}$.
Consider a controlled SDE 
\begin{align*}
\begin{cases}
  dX_s=b(s,X_s,u(s,X_s))ds+\sigma(s,X_s,u(s,X_s))dB_s,~~s\geq t+\tau,\\
  X_{t+\tau}=x.
\end{cases}
\end{align*} From the periodicity of coefficients in time, we have
\begin{align*}
X_{s+\tau}=&x+\int_{t+\tau}^{s+\tau}b(r,X_r,u(r,X_r))dr+\int_{t+\tau}^{s+\tau}\sigma(r,X_r,u(r,X_r))dB_r\\
=&x+\int_t^sb(r,X_{r+\tau},u(r,X_{r+\tau}))dr+\int_t^s\sigma(r,X_{r+\tau},u(r,X_{r+\tau}))d(\theta_\tau B)_r.
\end{align*}
By the uniqueness of solution of the above SDE, we deduce that
\begin{align*}
X_s^{t,x,u}\circ\theta_\tau=X_{s+\tau}^{t+\tau,x,u}.
\end{align*}
It follows that for any $s\geq t$
\begin{align*}
  T_{s+\tau}^{t+\tau,u(\cdot)}[\varphi](x)=T_s^{t,u(\cdot)}[\varphi](x).
\end{align*}  
\end{proof}

Theorem \ref{yasuo} leads to the convergence of iterations of the controlled Markov process with the help of Lemma \ref{lemma 2.5} and the Chapman formula. To see this, for any $u(\cdot)\in\mathcal{A}_t$ and $n\geq 1$ 
\begin{align*}
     \left\|T^{t,u(\cdot)}_{t+n\tau}[\varphi]\right\|_\beta
   =&\sup_{x\neq y}\frac{\left|T^{t,u(\cdot)}_{t+n\tau}[\varphi](x)-T^{t,u(\cdot)}_{t+n\tau}[\varphi](y)\right|}{d_\beta(x,y)}\\
   =&\sup_{x\neq y}\frac{\left|T^{t,u(\cdot)}_{t+\tau}\circ T^{t+\tau,u(\cdot)}_{t+n\tau}[\varphi](x)-T^{t,u(\cdot)}_{t+\tau}\circ T^{t+\tau,u(\cdot)}_{t+n\tau}[\varphi](y)\right|}{d_\beta(x,y)}\\
\leq&\bar\alpha\sup_{x\neq y}\frac{\left|T^{t+\tau,u(\cdot)}_{t+n\tau}[\varphi](x)-T^{t+\tau,u(\cdot)}_{t+n\tau}[\varphi](y)\right|}{d_\beta(x,y)}\\
=&\bar\alpha\sup_{x\neq y}\frac{\left|T^{t,u(\cdot)}_{t+(n-1)\tau}[\varphi](x)-T^{t,u(\cdot)}_{t+(n-1)\tau}[\varphi](y)\right|}{d_\beta(x,y)}\\
=&\bar\alpha\left\|T^{t,u(\cdot)}_{t+(n-1)\tau}[\varphi]\right\|_\beta.
\end{align*}
Thus, by iteration we have
\begin{align*}
  \left\|T^{t,u(\cdot)}_{t+n\tau}[\varphi]\right\|_\beta\leq\bar\alpha^n\|\varphi\|_\beta.
\end{align*}

Since the coefficients are assumed to be periodic, we only consider the case where the initial time $t\in[0,\tau)$.
For any $s\in[\tau,2\tau)$, $n\geq0$, and $u(\cdot)\in\mathcal{U}_t$
\begin{align*}
 &\sup_{\varphi:\|\varphi\|_\beta\leq 1}\left|T_{s+(n+1)\tau}^{t,u(\cdot)}[\varphi](x)-T_{s+n\tau}^{t,u(\cdot)}[\varphi](x)\right|\\
\leq&\sup_{\varphi:\|\varphi\|_\beta\leq 1}E\left|T^{s,u(\cdot)}_{s+(n+1)\tau}[\varphi](X_s^{t,x,u})-T^{s,u(\cdot)}_{s+n\tau}[\varphi](X_s^{t,x,u})\right|.
\end{align*}
However, for any $y\in\mathbb{R}^d$,
\begin{align*}
\left|T^{s,u(\cdot)}_{s+(n+1)\tau}[\varphi](y)-T^{s,u(\cdot)}_{s+n\tau}[\varphi](y)\right|
\leq&E\left|T^{s+\tau,u(\cdot)}_{s+(n+1)\tau}[\varphi](X_{s+\tau}^{s,y,u})-T^{s+\tau,u(\cdot)}_{s+(n+1)\tau}[\varphi](y)\right|\\
\leq&\bar\alpha^n\|\varphi\|_\beta\left(2+\beta E|X_{s+\tau}^{s,y,u}|+\beta |y|\right)\\
\leq&\bar{\alpha}^n\|\varphi\|_\beta C(1+|y|),
\end{align*}
which implies that 
\begin{align}\label{yasuo 0}
 \sup_{\varphi:\|\varphi\|_\beta\leq 1}\left|T_{s+(n+1)\tau}^{t,u(\cdot)}[\varphi](x)-T_{s+n\tau}^{t,u(\cdot)}[\varphi](x)\right|
\leq\bar{\alpha}^nC(1+|x|).
\end{align}
Hence, we can obtain the result that there exists a constant $\vartheta[\varphi]$, which in fact is the integral of $\varphi$ with respect to a probability measure $\vartheta$, such that
\begin{align*}
\sup_{\varphi:\|\varphi\|_\beta\leq1}\left|T_{s+n\tau}^{t,u(\cdot)}[\varphi](x)-\vartheta[\varphi]\right|\leq\bar{\alpha}^nC(1+|x|)\to0,~~as~n\to\infty.
\end{align*}
It is worth pointing out that $\vartheta$ is independent of the initial conditions, including the initial states and the initial times. In fact, for different $x$ and $y$
\begin{align*}
|\vartheta[\varphi](x)-\vartheta[\varphi](y)|=&\lim_{n\to\infty}\left|T^{t,u(\cdot)}_{s+n\tau}[\varphi](x)-T^{t,u(\cdot)}_{s+n\tau}[\varphi](y)\right|\\
\leq&\lim_{n\to\infty}\bar\alpha^n\|\varphi\|_\beta\left(2+\beta E|X_s^{t,x,u}|+\beta E|X_s^{t,y,u}|\right)\\
\leq&\lim_{n\to\infty}\bar\alpha^n\|\varphi\|_\beta C(1+|x|+|y|)=0,
\end{align*}
and for any $t'\in[t,t+\tau)$,
\begin{align*}
|\vartheta_t[\varphi](x)-\vartheta_{t'}[\varphi](x)|=&\lim_{n\to\infty}\left|T^{t,u(\cdot)}_{s+n\tau}[\varphi](x)-T^{t',u(\cdot)}_{s+n\tau}[\varphi](x)\right|\\
\leq&\lim_{n\to\infty} E\left|T^{t',u(\cdot)}_{s+n\tau}[\varphi](X_{t'}^{t,x,u})-T^{t',u(\cdot)}_{s+n\tau}[\varphi](x)\right|\\
\leq&\lim_{n\to\infty}\bar\alpha^{n-1}\|\varphi\|_\beta C(1+|x|)=0.
\end{align*}

Now, we want to extend the results to the functions depending on the time variables and control processes. We in fact have proved the following, for the running function $f$ and any $s\in[\tau,2\tau)$
\begin{align}\label{invariant measure}
\left|E[f(s+n\tau,X_{s+n\tau}^{t,x,u},u(s+n\tau,X_{s+n\tau}^{t,x,u}))]-\rho_s^u[f] \right| \leq\bar{\alpha}^nC(1+|x|)\to0,~~as~n\to\infty,
\end{align}
where $\rho_s^u[f]:=\vartheta[f(s,\cdot,u(s,\cdot))]$ is $\tau$-periodic.

By continuity, we find that the function $s\mapsto\rho^u_s[f]$ is right-continuous as
\begin{align*}
 &|\rho^u_s[f]-\rho^u_r[f]|\\
=&\left|\lim_{n\to\infty} E\left[f(s+n\tau,X_{s+n\tau}^{t,x,u},u(s+n\tau,X_{s+n\tau}^{t,x,u}))\right]\right.\\
&~~~~~~~~~~\left.-\lim_{n\to\infty} E\left[f(r+n\tau,X_{r+n\tau}^{t,x,u},u(r+n\tau,X_{r+n\tau}^{t,x,u}))\right]\right|\\
\leq&\lim_{n\to\infty} E\left|f(s+n\tau,X_{s+n\tau}^{t,x,u},u(s+n\tau,X_{s+n\tau}^{t,x,u}))-f(r+n\tau,X_{r+n\tau}^{t,x,u},u(r+n\tau,X_{r+n\tau}^{t,x,u}))\right|\\
 \leq&\lim_{n\to\infty} E\left|f(s+n\tau,X_{s+n\tau}^{t,x,u},u(s+n\tau,X_{s+n\tau}^{t,x,u}))-f(s+n\tau,X_{r+n\tau}^{t,x,u},u(s+n\tau,X_{r+n\tau}^{t,x,u}))\right|\\
     +&\lim_{n\to\infty} E\left|f(s+n\tau,X_{r+n\tau}^{t,x,u},u(s+n\tau,X_{r+n\tau}^{t,x,u}))-f(r+n\tau,X_{r+n\tau}^{t,x,u},u(r+n\tau,X_{r+n\tau}^{t,x,u}))\right|\\
 \to&0,~as~~s\downarrow r,
\end{align*}
and 
\begin{align*}
    |\rho_r^u[f]|&\leq\lim_{n\to\infty}E\left|f(r+n\tau,X_{r+n\tau}^{t,x,u},u(r+n\tau,X_{r+n\tau}^{t,x,u}))\right|\\
    &\leq\lim_{n\to\infty}C(1+E|X_{r+n\tau}^{t,x,u}|)\leq C,
\end{align*}where $C>0$ is independent of the control process $u(\cdot)$.

Based on the above arguments, we can prove the following result.
\begin{theorem}\label{exist}
 Under Assumption \ref{assume}, for each $u(\cdot)\in\mathcal{U}_t$, the integration
\begin{align*}
\rho^u[f]:=\frac{1}{\tau}\int_0^\tau\rho^u_s[f]ds,
\end{align*}
satisfies
\begin{equation}\label{rho}
  \rho^u[f]=\lim_{T\to\infty}\frac{1}{T}\int_t^{t+T}E[f(s,X_s^{t,x,u},u(s,X_s^{t,x,u}))]ds,
\end{equation}
and 
\begin{align}\label{dominate}
    \left|\lim_{n\to\infty}\int_t^{n\tau}(E[f(s,X_s^{t,x,u},u(s,X_s^{t,x,u}))]-\rho^u[f])ds\right|\leq C(1+|x|).
\end{align}
\end{theorem}
\begin{proof} 
Based on (\ref{invariant measure}), we see that for any $\varepsilon>0$, there exists a uniform positive integer $N_\varepsilon$ such that for any $s\in[t,t+\tau)$ and $n>N_\varepsilon$
\begin{align*}
\left|\frac{1}{n}\sum_{k=0}^{n-1} E[f(s+k\tau,X_{s+k\tau}^{t,x,u},u(s+k\tau,X_{s+k\tau}^{t,x,u}))]-\rho^u_s[f]\right|<\varepsilon.
\end{align*}
Hence, for any $n>N_\varepsilon$, we get
\begin{align*}
 &\left|\frac{1}{n}\int_t^{t+n\tau} E[f(s,X_s^{t,x,u},u(s,X_s^{t,x,u}))]ds-\tau\rho^u[f]\right|\\
=&\left|\frac{1}{n}\int_t^{t+\tau}\sum_{k=0}^{n-1} E[f(s+k\tau,X_{s+k\tau}^{t,x,u},u(s+k\tau,X_{s+k\tau}^{t,x,u}))]ds-\int_t^{t+\tau}\rho^u_s[f]ds\right|\\
<&\tau\varepsilon.
\end{align*}
Then, the first result of this theorem follows from some standard arguments.

For (\ref{dominate}),
\begin{align*}
    &\left|\lim_{n\to\infty}\int_t^{n\tau}(E[f(s,X_s^{t,x,u},u(s,X_s^{t,x,u}))]-\rho^u[f])ds\right|\\
    \leq&\int_t^\tau\left|E[f(s,X_s^{t,x,u},u(s,X_s^{t,x,u}))]-\rho^u[f]\right|ds\\
    &~~~~~+\lim_{n\to\infty}\int_\tau^{n\tau}\left|E[f(s,X_s^{t,x,u},u(s,X_s^{t,x,u}))]-\rho^u_s[f]\right|ds\\
    \leq&C(1+|x|)+\frac{1}{1-\bar\alpha}C(1+|x|),
\end{align*}
where the constant $C$ is independent of the control process $u(\cdot)$.
\end{proof}

As a consequence, we have the following results from the Lipschitz continuity of $f$ in $x$. 
\begin{theorem}\label{theorem ergodic}
Under Assumption \ref{assume}, there exists a constant $\rho$ such that 
\begin{align*}
\rho=\sup_{u(\cdot)\in\mathcal{U}_t}\lim_{T\to\infty}\frac{1}{T}E\int_t^{t+T}f(s,X^{t,x,u}_s,u(s,X^{t,x,u}_s))ds.
\end{align*}
\end{theorem}

This method also works for the autonomous system. For each fixed $u(\cdot)\in\mathcal{U}$, we can artificially define a period $\tau>0$, and then prove that there exists a function $s\mapsto\rho_s^u$ which is independent of the initial conditions such that
  \begin{align*}
      \rho_s^u=\lim_{n\to\infty}E[f(X_{s+n\tau}^{x,u},u(X_{s+n\tau}^{x,u}))].
  \end{align*}
  In addition, we can prove that $\rho_s^u=\rho^u_t$ for any $s,t\in[0,\tau)$ by homogeneity. Hence, we can obtain from (\ref{rho}) that $\rho^u=\frac{1}{\tau}\int_0^\tau\rho_s^uds$ trivially exists and there exists a $\rho$ such that 
  \begin{align*}
  \rho=\sup_{u(\cdot)\in\mathcal{U}}\rho^u=\sup_{u(\cdot)\in\mathcal{U}}\lim_{T\to\infty}\frac{1}{T}E\int_0^Tf(X_s^{x,u},u(X_s^{x,u}))ds.
  \end{align*}

\section{Dynamic programming principle}\label{dynamic programming principle}
In Section \ref{convergency results}, we have proved that the average stochastic optimal control problem is ergodic. The aim of the next step is to find the optimal controls of the stochastic control problem. Note that it is not easy to find the dynamic programming principle satisfied by the value function (\ref{value function}). We should point out that Theorem \ref{theorem ergodic} does not imply the existence of optimal controls. For any $0<\varepsilon<1$, there exists $u^\varepsilon(\cdot)\in\mathcal{U}_t$ such that 
\begin{align*}
    \rho^{u^\varepsilon}:=\rho^\varepsilon>\rho-\varepsilon.
\end{align*}
To treat the possibility that the optimal controls do not exist, we construct an auxiliary $\tau$-periodic function as
\begin{align*}
  J_\varepsilon(t,x,u(\cdot))=\lim_{n\to\infty}\int_t^{n\tau} E\left[f(s,X_s^{t,x,u},u(s,X_s^{t,x,u}))-\rho^\varepsilon\right]ds,
\end{align*}
and the auxiliary value function is 
\begin{align*}
W_\varepsilon(t,x)=\sup_{u(\cdot)\in\mathcal{U}_t^\varepsilon}J_\varepsilon(t,x,u(\cdot)),
\end{align*}
where $\mathcal{U}_t^\varepsilon\subset\mathcal{U}_t$ is the collection of all $u(\cdot)\in\mathcal{U}_t$ such that $\rho^u\leq\rho^\varepsilon$. 
Note, $J_\varepsilon(t,x,u(\cdot))$ is only finite in a subset $\mathcal{U}_t^{\varepsilon,0}:=\{u(\cdot)\in\mathcal{U}_t^\varepsilon;\rho^u=\rho^\varepsilon\}$, otherwise, $J_\varepsilon(t,x,u(\cdot))=-\infty$. 
This can be seen since for any $u(\cdot)\in\mathcal{U}_t$
\begin{align*}
J_\varepsilon(t,x,u(\cdot))=&\lim_{n\to\infty}\int_t^{n\tau}\left(E[f(s,X_s^{t,x,u},u(s,X_s^{t,x,u}))]-\rho^u\right)ds+\int_t^\infty(\rho^u-\rho^\varepsilon)ds,
\end{align*}
which means
\begin{align*}
-C(1+|x|)+\int_t^\infty(\rho^u-\rho^\varepsilon)ds\leq J_\varepsilon(t,x,u(\cdot))
\leq C(1+|x|)+\int_t^\infty(\rho^u-\rho^\varepsilon)ds.
\end{align*}
So, if $u(\cdot)\in\mathcal{U}_t^\varepsilon\backslash\mathcal{U}_t^{\varepsilon,0}$, we have 
\begin{align*}
    \int_t^\infty (\rho^u-\rho^\varepsilon) ds=-\infty.
\end{align*}
Thus, the auxiliary function can be rewritten as 
\begin{align}\label{4.1a}
  W_\varepsilon(t,x)=\sup_{u(\cdot)\in\mathcal{U}_t^{\varepsilon,0}}J_\varepsilon(t,x,u(\cdot)).
\end{align}
Hence, we have the linear growth of $W_\varepsilon(t,x)$ in $\mathbb{R}^d$ that
\begin{align}\label{linear}
    |W_\varepsilon(t,x)|\leq C(1+|x|).
\end{align}

\begin{remark}
If the optimal controls exist, one can choose $u^*(\cdot)$ with $\rho^{u^*}=\rho$ and construct a function as
\begin{align*}
W(t,x)=\sup_{u(\cdot)\in\mathcal{U}_t}\lim_{n\to\infty}\int_t^{n\tau}(E[f(s,X_s^{t,x,u},u(s,X_s^{t,x,u}))]-\rho)ds.
\end{align*}
\end{remark}

In the following, we will give the continuity of the auxiliary value function $W_\varepsilon(t,x)$ with respect to $x$ and $t$, respectively. We only consider the case that $t\in[0,\tau)$. 

\begin{lemma} 
Under Assumption \ref{assume}, the $\tau$-periodic function $W_\varepsilon(t,x)$ is continuous in intervals $\{[k\tau,(k+1)\tau)\times\mathbb{R}^d,k\geq 0\}$, and has a left-limit at $k\tau$ for each $k\geq1$.
\end{lemma} 
\begin{proof}
The proof is divided into three steps.

\emph{Step 1.} In the first part, we will prove that $W_\varepsilon(t,x)$ is continuous in $x$.\\
For different $x,x'\in\mathbb{R}^d$
\begin{align*}
W_\varepsilon(t,x)-W_\varepsilon(t,x')=&\sup_{u(\cdot)\in\mathcal{U}_t^{\varepsilon,0}}\lim_{n\to\infty}\int_t^{n\tau}E[f(s,X_s^{t,x,u},u(s,X_s^{t,x,u}))-\rho^\varepsilon]ds\\
&~~~-\sup_{u(\cdot)\in\mathcal{U}_t^{\varepsilon,0}}\lim_{n\to\infty}\int_t^{n\tau}E[f(s,X_s^{t,x',u},u(s,X_s^{t,x',u}))-\rho^\varepsilon]ds.
\end{align*}
For any $h>0$, there exists $u^h(\cdot)\in\mathcal{U}_t^{\varepsilon,0}$ such that
\begin{align*}
W_\varepsilon(t,x)<\lim_{n\to\infty}\int_t^{n\tau} E\left[f(s,X_s^{t,x,u^h},u^h(s,X_s^{t,x,u^h}))-\rho^\varepsilon\right]ds+h.
\end{align*}
Then,
\begin{align*}
&W_\varepsilon(t,x)-W_\varepsilon(t,x')\\
&<\lim_{n\to\infty}\int_t^{n\tau} E\left[f(s,X_s^{t,x,u^h},u^h(s,X_s^{t,x,u^h}))-f(s,X_s^{t,x',u^h},u^h(s,X_s^{t,x',u^h}))\right]ds+h. 
\end{align*}
Since for any fixed $u(\cdot)\in\mathcal{U}_t$ and any $(n+1)\tau\leq s<(n+2)\tau$
\begin{align*}
\left|E\left[f(s,X_s^{t,x,u},u(s,X_s^{t,x,u}))-f(s,X_s^{t,x',u},u(s,X_s^{t,x',u}))\right]\right|\leq\bar\alpha^nC(1+|x|+|x'|).
\end{align*}
Let $N=[\log_{\bar\alpha}(\frac{h(1-\bar\alpha)}{C\tau(1+|x|+|x'|)})]+1$ and $T=(N+1)\tau$. Then, we have
\begin{align}
\int_T^\infty\left|E\left[f(s,X_s^{t,x,u},u(s,X_s^{t,x,u}))-f(s,X_s^{t,x',u},u(s,X_s^{t,x',u}))\right]\right|ds<h.
\end{align} 
Thus, by (\ref{estimates}) we have 
\begin{align*}
W_\varepsilon(t,x)-W_\varepsilon(t,x')
\leq&h+\int_t^TE|f(s,X_s^{t,x,u^h},u^h(s,X_s^{t,x,u^h}))-f(s,X_s^{t,x',u^h},u^h(s,X_s^{t,x',u^h})|ds\\
+&\int_T^\infty\left|E\left[f(s,X_s^{t,x,u^h},u^h(s,X_s^{t,x,u^h}))-f(s,X_s^{t,x',u^h},u(s,X_s^{t,x',u^h}))\right]\right|ds\\
<&\frac{2C_f}{C_\sigma^2+2C_b}e^{(N+1)\tau(C_\sigma^2+2C_b)/2}|x-x'|+2h\\
 \leq&L|x-x'|+2h, 
\end{align*}
where $L=\frac{2C_f}{C_\sigma^2+2C_b}e^{\tau(C_\sigma^2+2C_b)}(\frac{C\tau(1+|x|+|x'|)}{h(1-\bar\alpha)})^{\frac{\tau(C_\sigma^2+2C_b)/2}{\ln 1/\bar\alpha}}$. If $|x-x'|$ is small enough, we can also see that the constant $L$ is independent of $x'$ since $(1+|x|+|x'|)\leq2(1+|x|)$. Thus, we can obtain $W_\varepsilon(t,x)-W_\varepsilon(t,x')<3h$ when $|x-x'|$ is small enough.

On the other hand, there exists $u^h(\cdot)\in\mathcal{U}_t$ such that
\begin{align*}
&W_\varepsilon(t,x)-W_\varepsilon(t,x')\\
>&\lim_{n\to\infty}\int_t^{n\tau} E\left[f(s,X_s^{t,x,u^h},u^h(s,X_s^{t,x,u^h}))-f(s,X_s^{t,x',u^h},u^h(s,X_s^{t,x',u^h}))\right]ds-h.
\end{align*}
Similarly, we have $W_\varepsilon(t,x)-W_\varepsilon(t,x')>-3h$ when $|x-x'|$ is small enough. Thus, we find that $W_\varepsilon(t,x)$ is continuous in $x$.
We should point out that the continuity in $x$ is uniform in $\varepsilon$, though not uniform in $x$.

\emph{Step 2.} Now we are ready to prove that the function $W_\varepsilon(t,x)$ is continuous in $t$.\\
Given $0\leq t\leq t+\delta<\tau$ and $x\in\mathbb{R}^d$, there exists $\tilde u^t(\cdot)\in\mathcal{U}_t^{\varepsilon,0}$ such that
\begin{align*}
    W_\varepsilon(t,x)<J_\varepsilon(t,x,\tilde u^t(\cdot))+h.
\end{align*}
Let $\tilde u^{t+\delta}(\cdot,x)$ be the restriction of $\tilde u^t(\cdot,x)$, the corresponding measurable function of $\tilde u^t(\cdot)$, on $[t+\delta,\infty)$, then the corresponding control process $\tilde u^{t+\delta}(\cdot)\in\mathcal{U}_{t+\delta}^{\varepsilon,0}$. So from \emph{Step 1} we have
\begin{align*}
&W_\varepsilon(t,x)-W_\varepsilon(t+\delta,x)<J_\varepsilon(t,x,\tilde u^t(\cdot))-J_\varepsilon(t+\delta,x,\tilde u^{t+\delta}(\cdot))+h\\
=&\lim_{n\to\infty}\int_t^{n\tau}\left(E[f(s,X_s^{t,x,\tilde u},\tilde u(s,X_s^{t,x,\tilde u}))]-\rho^\varepsilon\right)ds+h\\
&-\lim_{n\to\infty}\int_{t+\delta}^{n\tau}\left(E[f(s,X_s^{t+\delta,x,\tilde u},\tilde u(s,X_s^{t+\delta,x,\tilde u}))]-\rho^\varepsilon\right)ds\\
=&\int_t^{t+\delta}\left(E[f(s,X_s^{t,x,\tilde u},\tilde u(s,X_s^{t,x,\tilde u}))]-\rho^\varepsilon\right)ds+h\\
&+\lim_{n\to\infty}\int_{t+\delta}^{n\tau}\left(E[f(s,X_s^{t,x,\tilde u},\tilde u(s,X_s^{t,x,\tilde u}))]-E[f(s,X_s^{t+\delta,x,\tilde u},\tilde u(s,X_s^{t+\delta,x,\tilde u}))]\right)ds\\
<&\int_t^{t+\delta}\left|E[f(s,X_s^{t,x,\tilde u},\tilde u(s,X_s^{t,x,\tilde u}))]-\rho^\varepsilon\right|ds+LE|X_{t+\delta}^{t,x,\tilde u}-x|+3h,
\end{align*}
where $L\leq C+E(1+|x|+|X_{t+\delta}^{t,x,\tilde u}|)^p\leq C(1+|x|^p)$, for some $C,p>0$.
As a result, we have 
\begin{align*}
    W_\varepsilon(t,x)-W_\varepsilon(t+\delta,x)<C\delta^{1/2}+C_x\delta+3h.
\end{align*}
By symmetry and the arbitrariness of $h$, we have 
\begin{align*}
    |W_\varepsilon(t,x)-W_\varepsilon(t+\delta,x)|\leq C(\delta^{1/2}+\delta),
\end{align*}
where $C>0$ is independent of $\varepsilon$, but depends on $x$. So, we obtain the equicontinuity of the family of value functions $\{W_\varepsilon(t,x)\}_{\varepsilon>0}$. 

\emph{Step 3.}
Furthermore, from the definition of the auxiliary value function, we can claim that
\begin{align}\label{4.2a}
    W_\varepsilon(\tau-,x)=\sup_{u(\cdot)\in\mathcal{U}^\varepsilon_\tau}J_\varepsilon(\tau,x,u(\cdot)).
\end{align}
In fact, for any $h>0$ there exists $\tilde u(\cdot)\in\mathcal{A}^\varepsilon_\tau$ such that 
\begin{align*}
    J_\varepsilon(\tau,x,\tilde u(\cdot))>\sup_{u(\cdot)\in\mathcal{A}^\varepsilon_\tau}J_\varepsilon(\tau,x,u(\cdot))-h,
\end{align*}
and $\tilde u(\cdot)$ can be extended as a process in $\mathcal{U}_t^\varepsilon$, still denoted by $\tilde u(\cdot)$, then from \emph{Step 1} we have 
\begin{align*}
&\sup_{u(\cdot)\in\mathcal{A}^\varepsilon_\tau}J_\varepsilon(\tau,x,u(\cdot))-\lim_{t\to\tau-}W_\varepsilon(t,x)\\
   <&J_\varepsilon(\tau,x,\tilde u(\cdot))+h-\lim_{t\to\tau-}J_\varepsilon(t,x,\tilde u(\cdot))\\
   =&\lim_{n\to\infty}\int_\tau^{n\tau}\left(E[f(s,X_s^{\tau,x,\tilde u},\tilde u(s,X_s^{\tau,x,\tilde u}))]-\rho^\varepsilon\right)ds+h\\ 
    &-\lim_{t\to\tau-}\lim_{n\to\infty}\int_t^{n\tau}\left(E[f(s,X_s^{t,x,\tilde u},\tilde u(s,X_s^{t,x,\tilde u}))]-\rho^\varepsilon\right)ds\\
    =&\lim_{t\to\tau-}\lim_{n\to\infty}\int_\tau^{n\tau}\left(E[f(s,X_s^{\tau,x,\tilde u},\tilde u(s,X_s^{\tau,x,\tilde u}))]-E[f(s,X_s^{\tau,X_\tau^t,\tilde u},\tilde u(s,X_s^{\tau,X_\tau^t,\tilde u}))]\right)ds\\ 
    &-\lim_{t\to\tau-}\int_t^{\tau}\left(E[f(s,X_s^{t,x,\tilde u},\tilde u(s,X_s^{t,x,\tilde u}))]-\rho^\varepsilon\right)ds+h\\
    \leq&\lim_{t\to\tau-}LE|X_\tau^{t,x,\tilde u}-x|-\lim_{t\to\tau-}\int_t^\tau\left(E[f(s,X_s^{t,x,\tilde u},\tilde u(s,X_s^{t,x,\tilde u}))]-\rho^\varepsilon\right)ds+3h\\
    =&3h.
\end{align*}
Since $h>0$ is arbitrary, we have
\begin{align*}
    \sup_{u(\cdot)\in\mathcal{A}^\varepsilon_\tau}J_\varepsilon(\tau,x,u(\cdot))-\lim_{t\to\tau-}W_\varepsilon(t,x)\leq0.
\end{align*}
The inverse inequality can be similarly obtained. Since $\mathcal{A}_t\subset\mathcal{U}_t$ for any $t\geq 0$, then
\begin{align*}
W_\varepsilon(\tau,x)=\sup_{u(\cdot)\in\mathcal{U}_\tau}J_\varepsilon(\tau,x,\alpha(\cdot))\geq\sup_{u(\cdot)\in\mathcal{A}_\tau}J_\varepsilon(\tau,x,\alpha(\cdot))= W_\varepsilon(\tau-,x).
\end{align*}
This completes the proof.
\end{proof}

By the Arzal\`a-Ascoli theorem and some standard arguments, we have that there exists a subsequence $\{\varepsilon_k\}$, still denote $\{\varepsilon\}$, such that $\lim_{\varepsilon\to0}W_\varepsilon(t,x)$ exists and denotes by
\begin{align*}
    w(t,x)=\lim_{\varepsilon\to0}W_\varepsilon(t,x).
\end{align*}

\begin{lemma}
    The function $w(t,x)$ has the following properties:
    
    1) linear growth:
    \begin{align*}
        |w(t,x)|\leq C(1+|x|),
    \end{align*}

    2) continuity in each interval $[n\tau,(n+1)\tau)\times\mathbb{R}^d$ and left-limit $\lim_{t\to n\tau-}w(t,x)$ exists for any $n\geq 1$ and monotonicity:
    \begin{align*}
        w(n\tau-,x)\leq w(n\tau,x).
    \end{align*}
\end{lemma}

Then we will give the dynamic programming principle satisfied by the auxiliary function $w(t,x)$.
\begin{theorem}\label{dynamic programming}
The function $w(t,x)$ satisfies the dynamic programming equation: for any bounded stopping time $t<\Theta<\tau$ that 
\begin{align}\label{DPP}
w(t,x)=\sup_{u(\cdot)\in\mathcal{U}_t}E\left[\int_t^\Theta\left[f(s,X_s^{t,x,u},u(s,X_s^{t,x,u}))-\rho\right]ds+w(\Theta,X_\Theta^{t,x,u})\right].
\end{align}
\end{theorem}
\begin{proof}
The proof is divided into two steps.

\emph{Step 1.} In this part, we will prove that the function $w(t,x)$ satisfies the dynamic programming equation (\ref{DPP}) for any bounded stopping time $t\leq\Theta<\tau$ taking only finitely many values $\{t_1,t_2,\cdots,t_l\}$.

For any $h>0$, there exists $u^h(\cdot)\in\mathcal{U}_t$ such that 
\begin{align}\label{ll}
\begin{split}
&\sup_{u(\cdot)\in\mathcal{U}_t}E\left[w(\Theta,X_\Theta^{t,x,u})+\int_t^\Theta(f(s,X_s^{t,x,u},u(s,X_s^{t,x,u}))-\rho)ds\right]\\
 <&E\left[w(\Theta,X_\Theta^{t,x,u^h})+\int_t^\Theta(f(s,X_s^{t,x,u^h},u^h(s,X_s^{t,x,u^h}))-\rho)ds\right]+h. 
  \end{split}
\end{align} 
By the Markov inequality, we see that for any $s\geq t$ and constant $G>0$
\begin{align*}
P\{|X_s|^2\geq G\}\leq\frac{E|X_s|^2}{G}\leq\frac{L/\lambda+|x|^2}{G}.
\end{align*}
Let $\{I_{s,k}\}_{0\leq k\leq N}$ be a finite partition of $B_G=\{x\in\mathbb{R}^d;|x|^2\leq G\}$
such that for any given $x_k\in I_{s,k}$ and $y\in I_{s,k}$, we have for any $\varepsilon>0$,
\begin{align*}
W_\varepsilon(s,y)<W_\varepsilon(s,x_k)+h,
\end{align*}
and 
\begin{align*}
  &\lim_{n\to\infty}\int_s^{n\tau}E[f(r,X_r^{s,x_k,u^{s,k,h}},u^{s,k,h}(r,X_r^{s,x_k,u^{s,k,h}}))-\rho^\varepsilon]dr\\
<&\lim_{n\to\infty}\int_s^{n\tau}E[f(r,X_r^{s,y,u^{s,k,h}},u^{s,k,h}(r,X_r^{s,y,u^{s,k,h}}))-\rho^\varepsilon]dr+h,
\end{align*}
where $u^{s,k,h}(\cdot)\in\mathcal{U}_s^\varepsilon$ is such that 
\begin{align*}
W_\varepsilon(s,x_k)<\lim_{n\to\infty}\int_s^{n\tau}E[f(r,X_r^{s,x_k,u^{s,k,h}},u^{s,k,h}(r,X_r^{s,x_k,u^{s,k,h}}))-\rho^\varepsilon]dr+h.
\end{align*}
Then, for any $y\in I_{s,k}$,
\begin{align}\label{4.4a}
W_\varepsilon(s,y)<\lim_{n\to\infty}\int_s^{n\tau}(E[f(r,X_r^{s,y,u^{s,k,h}},u^{s,k,h}(r,X_r^{s,y,u^{s,k,h}}))]-\rho^\varepsilon)dr+3h.
\end{align}
However, for any $y\notin B_G$, by the linear growth property of $W_\varepsilon(s,x)$ we have
\begin{align*}
  |W_\varepsilon(s,y)-W_\varepsilon(s,x_0)|\leq C(1+|x_0|+|y|),
\end{align*}
and 
\begin{align*}
&\bigg|\lim_{n\to\infty}\int_s^{n\tau}E[f(r,X_r^{s,x_0,u^{s,0,h}},u^{s,0,h}(r,X_r^{s,x_0,u^{s,0,h}}))-\rho^\varepsilon]dr\\
&-\lim_{n\to\infty}\int_s^{n\tau}E[f(r,X_r^{s,y,u^{s,0,h}},u^{s,0,h}(r,X_r^{s,y,u^{s,0,h}}))-\rho^\varepsilon]dr\bigg|\leq C(1+|y|+|x_0|).
\end{align*}
So, we have for any $y\notin B_G$
\begin{align*}
W_\varepsilon(s,y)\leq \lim_{n\to\infty}\int_s^{n\tau}E[f(r,X_r^{s,y,u^{s,0,h}},u^{s,0,h}(r,X_r^{s,y,u^{s,0,h}}))-\rho^\varepsilon]dr\\
+2C(1+|y|+|x_0|)+h.
\end{align*}
Let
\begin{align*}
u^s(\cdot)=\sum_{k=1}^NI_{\{X_s\in I_{s,k}\}}u^{s,k,h}(\cdot)+u^{s,0,h}(\cdot)I_{\{X_s\in I_{s,0}\cup B_G^c\}},
\end{align*}
and
\begin{align}\label{tilde u}
\tilde{u}_r(\cdot)=
\begin{cases}
u^h_r,~r\in[t,\Theta],\\
\sum_{m=1}^lI_{\{\Theta=t_m\}}u^{t_m}_r,~r>\Theta.
\end{cases}
\end{align}
Then, it is easy to check that $\tilde u(\cdot)\in\mathcal{U}_t^\varepsilon$ is an admissible control. 
By (\ref{ll}) and (\ref{4.4a}), we have, for some large enough $G$
\begin{align*}
  &E\left[\int_t^\Theta(f(r,X_r^{t,x,u^h},u^h(r,X_r^{t,x,u^h}))-\rho^\varepsilon)dr+W_\varepsilon(\Theta,X_\Theta^{t,x,u^h})\right]+h\\
 <&\sum_{m=1}^lP(\Theta=t_m)E\left[\int_t^\Theta(f(r,X_r^{t,x,u^h},u^h(r,X_r^{t,x,u^h}))-\rho^\varepsilon)dr\right.\\
  &~~~~~~~~~~~~~\left.+\lim_{n\to\infty}\int_\Theta^{n\tau} E[f(r,X_r^{\Theta,X_\Theta,u^\Theta},u^\Theta(r,X_r^{\Theta,X_\Theta,u^\Theta}))-\rho^\varepsilon|\mathcal{F}_\Theta]dr~\right.\\
  &~~~~~~~~~~~~~\left.+C(1+|X_\Theta|+|x_0|)I_{\{X_\Theta\notin B_G\}}\bigg|\Theta=t_m\right]+3h\\
 <&\sum_{m=1}^lP(\Theta=t_m)\left(\int_t^{t_m}E[f(r,X_r^{t,x,\tilde u},\tilde u(r,X_r^{t,x,\tilde u}))-\rho^\varepsilon]dr\right. \\
  &~~~~~~~~~~~~~~~~~~~~~~~~~~+\left.\lim_{n\to\infty}\int_{t_m}^{n\tau} E[f(r,X_r^{t,x,\tilde u},\tilde u(r,X_r^{t,x,\tilde u}))-\rho^\varepsilon]dr\right)+4h\\
 \leq&J_\varepsilon(t,x,\tilde u(\cdot))+4h
 \leq W_\varepsilon(t,x)+4h.
\end{align*}
It follows that for any sufficiently small $\varepsilon>0$, we have 
\begin{align*}
    &\sup_{u(\cdot)\in\mathcal{U}_t}E\left[w(\Theta,X_\Theta^{t,x,u})+\int_t^\Theta(f(s,X_s^{t,x,u},u(s,X_s^{t,x,u}))-\rho)ds\right]\\
    <&W_\varepsilon(t,x)+5h+E\left[\int_t^\Theta(\rho^\varepsilon-\rho)dr+w(\Theta,X_\Theta^{t,x,u^h})-W_\varepsilon(\Theta,X_\Theta^{t,x,u^h})\right].
\end{align*}
Since $h>0$ is arbitrary and let $\varepsilon\to0$, we have 
\begin{align}\label{r}
  \sup_{u(\cdot)\in\mathcal{U}_t}E\left[w(\Theta,X_\Theta^{t,x,u})+\int_t^\Theta(f(r,X_r^{t,x,u},u(r,X_r^{t,x,u}))-\rho)dr\right] \leq w(t,x).
\end{align}
 
On the other hand, there exists a control $u^h(\cdot)\in\mathcal{U}_t^\varepsilon$ such that
\begin{align*}
\begin{split}
W_\varepsilon(t,x)<&\lim_{n\to\infty}\int_t^{n\tau} E\left[f(r,X_r^{t,x,u^h},u^h(r,X_r^{t,x,u^h}))-\rho^\varepsilon\right]dr+h\\
=&h+\lim_{n\to\infty}\sum_{m=1}^lP(\Theta=t_m)E\left[\int_t^{t_m} (f(r,X_r^{t,x,u^h},u^h(r,X_r^{t,x,u^h}))-\rho^\varepsilon)dr\right.\\
&~~~~~~~~~~~~~~~~~~~~~~~~\left.+\int_{t_m}^{n\tau}E\left[f(r,X_r^{t,x,u^h},u^h(r,X_r^{t,x,u^h}))-\rho^\varepsilon|\mathcal{F}_{t_m}\right]dr\right]\\
=&h+\lim_{n\to\infty}\sum_{m=1}^lP(\Theta=t_m)E\left[\int_t^\Theta(f(r,X_r^{t,x,u^h},u^h(r,X_r^{t,x,u^h}))-\rho^\varepsilon)dr\right.\\
&~~~~~~~~~~~~~~~~\left.+\int_\Theta^{n\tau} E\left[f(r,X_r^{t,x,u^h},u^h(r,X_r^{t,x,u^h}))-\rho^\varepsilon|\mathcal{F}_\Theta\right]dr|\Theta=t_m\right]\\ 
\leq&h+E\left[\int_t^\Theta(f(r,X_r^{t,x,u^h},u^h(r,X_r^{t,x,u^h}))-\rho^\varepsilon)dr+W_\varepsilon(\Theta,X_\Theta^{t,x,u^h})\right].
\end{split}
\end{align*}
In the limit of $\varepsilon\to0$, we have
\begin{align*}
    w(t,x)<&h+E\left[\int_t^\Theta(f(r,X_r^{t,x,u^h},u^h(r,X_r^{t,x,u^h}))-\rho)dr+w(\Theta,X_\Theta^{t,x,u^h})\right]\\
    \leq&h+\sup_{u(\cdot)\in\mathcal{U}_t}E\left[\int_t^\Theta(f(r,X_r^{t,x,u},u(r,X_r^{t,x,u}))-\rho)dr+w(\Theta,X_\Theta^{t,x,u})\right].
\end{align*}
Since $\varepsilon\to0$, we have 
\begin{align}\label{l}
    w(t,x)\leq\sup_{u(\cdot)\in\mathcal{U}_t}E\left[\int_t^\Theta(f(r,X_r^{t,x,u},u(r,X_r^{t,x,u}))-\rho)dr+w(\Theta,X_\Theta^{t,x,u})\right].
\end{align}
Combing (\ref{r}) and (\ref{l}), we have that the dynamic programming equation (\ref{DPP}) is satisfied when $\Theta$ takes finitely many values.

\emph{Step 2}. For a general stopping time $t\leq\Theta<\tau$, let 
\begin{align*}
\Theta_n=\sum_{k=2^nt}^{2^n\tau-1}\frac{k}{2^n}I_{\{\frac{k}{2^n}\leq\Theta< \frac{k+1}{2^n}\}}.
\end{align*}
Then $\Theta_n$ is also a bounded $\mathcal{F}^s_t$-stopping time that take only finitely many values and $\Theta_n\leq\Theta<\Theta_n+\frac{1}{2^n}$.

Note that
\begin{align}\label{ww}
  \begin{split}
&\left|\sup_{u(\cdot)\in\mathcal{U}_t}E\left[w(\Theta,X_\Theta^{t,x,u})+\int_t^\Theta(f(r,X_r^{t,x,u},u(r,X_r^{t,x,u}))-\rho)dr\right]\right.\\
&\left.-\sup_{u(\cdot)\in\mathcal{U}_t}E\left[w(\Theta_n,X_{\Theta_n}^{t,x,u})+\int_t^{\Theta_n}(f(r,X_r^{t,x,u},u(r,X_r^{t,x,u}))-\rho)dr\right]\right|\\
\leq&\sup_{u(\cdot)\in\mathcal{U}_t}E\left|w(\Theta,X_\Theta^{t,x,u})-w(\Theta_n,X_{\Theta_n}^{t,x,u})\right|\\
&+\sup_{u(\cdot)\in\mathcal{U}_t}E\left[\int_{\Theta_n}^\Theta\left|f(r,X_r^{t,x,u},u(r,X_r^{t,x,u}))-\rho\right|dr\right],
  \end{split}
\end{align}
the first part of the right-hand side of the inequality converges to zero as $n\to \infty$ because of the dominated convergence theorem (dominated by $C(1+\sup_{s\in[t,\tau]}|X_s^{t,x,u}|)$), and the second part converges to zero as $n\to\infty$ because of the monotone convergence theorem. Finally, the result of this theorem follows from the above convergence and the result in Step 2.
\end{proof}

\section{HJB equation}\label{HJB}
In this section, we mainly consider the periodic solution of the following periodic HJB equation:
\begin{equation}\label{PDE}
  \begin{cases}
    \partial_t\varphi(t,x)+\sup\limits_{u\in U}H(t,x,u,D\varphi,D^2\varphi)=\rho,~~(t,x)\in[0,\tau)\times\mathbb{R}^d,\\
    \varphi(0,\cdot)=\varphi(\tau,\cdot),
  \end{cases}
\end{equation}
where the Hamilton function $H$ is defined as 
\begin{align*}
H(t,x,u,D\varphi,D^2\varphi)
=f(t,x,u)+D\varphi(t,x)\cdot b(t,x,u)+\frac{1}{2}tr(\sigma\sigma'(t,x,u)D^2\varphi(t,x)).
\end{align*}
It is worth pointing out that the ergodic HJB equation (\ref{PDE}) is very different from the HJB equation of finite time stochastic control problems which has a terminal condition, though they have same PDE.
This is actually an infinite horizon problem, which can be seen that the terminal condition is forgotten during the process of the ergodicity taking terminal time to infinity. We will prove that the solution of the equation is the value function $w$ of the ergodic control problem. 
As far as we know, this is the first attempt to consider infinite horizon non-autonomous HJB equations.

Since we can only prove that the auxiliary function $w(t,x)$ is continuous, we cannot prove that the pair $(w,\rho)$ is a classical solution of the HJB equation (\ref{PDE}). Thus, a viscosity solution as a kind of weak solution is considered. Our objective is to prove that the pair $(w,\rho)$ is a viscosity solution of (\ref{PDE}).

\begin{definition}
1) A pair $(w,\rho)$ with function $w(t,x)\in C([0,\tau)\times\mathbb{R}^d)$ and constant $\rho$ is called a viscosity subsolution of HJB equation (\ref{PDE}) if $w(0,x)=w(\tau,x)$ for any $x\in\mathbb{R}^d$, and if for any $\varphi\in C^{1,2}([0,\tau]\times\mathbb{R}^d)$, whenever $w(t,x)-\varphi(t,x)$ attains a local maximum at $(t_0,x_0)\in[0,\tau)\times\mathbb{R}^d$, we have
\begin{align*}
  \partial_t\varphi(t_0,x_0)+\sup\limits_{u\in U}H(t_0,x_0,u,D\varphi,D^2\varphi)\geq\rho.
\end{align*}
2) A pair $(w,\rho)$ is a viscosity supersolution of the HJB equation (\ref{PDE}) if $w(0,x)=w(\tau,x)$ for any $x\in\mathbb{R}^d$, and if for any $\varphi\in C^{1,2}([0,\tau]\times\mathbb{R}^d)$, whenever $w(t,x)-\varphi(t,x)$ attains a local minimum at $(t_0,x_0)\in[0,\tau)\times\mathbb{R}^d$, we have
\begin{align*}
  \partial_t\varphi(t_0,x_0)+\sup\limits_{u\in U}H(t_0,x_0,u,D\varphi,D^2\varphi)\leq\rho.
\end{align*}
3) A pair $(w,\rho)$ is a viscosity solution of (\ref{PDE}) if it is both a viscosity subsolution and a viscosity supersolution.
\end{definition}

In the following theorem, we will prove that the HJB equation (\ref{PDE}) has a viscosity solution.

\begin{theorem}\label{existence}
The pair $(w,\rho)$ is a viscosity solution of the HJB equation (\ref{PDE}).
\end{theorem}
\begin{proof}
\emph{Viscosity supersolution property.} 
Recall the dynamic programming equation (\ref{DPP}), we have for any stopping time $t\leq \Theta< \tau$ and $u(\cdot)\in\mathcal{U}_t$
\begin{align*}
w(t,x)\geq E\left[w(\Theta,X_\Theta^{t,x,u})+\int_t^\Theta(f(s,X_s^{t,x,u},u(s,X_s^{t,x,u}))-\rho)ds\right].
\end{align*}
For any $\phi\in C^{1,2}([0,\tau]\times\mathbb{R}^d)$, suppose that $w(t,x)-\phi(t,x)$ attains a local minimum at $(t_0,x_0)\in[0,\tau)\times\mathbb{R}^d$ in its neighbourhood $B(t_0,x_0)=[t_0,t_0+\delta)\times B_\varepsilon(x_0)\subset[0,\tau)\times\mathbb{R}^d$ for some $\varepsilon,\delta>0$, where $B_\varepsilon(x_0)=\{y\in\mathbb{R}^d:|y-x_0|<\varepsilon\}$. Let 
\begin{align*}
    \Theta^{t_0,x_0,u}=\inf\{s\geq t_0,X_s^{t_0,x_0,u}\notin B_\varepsilon(x_0)\},
\end{align*}
then we have
\begin{align*}
0\leq&E\left[w(\Theta,X_\Theta^{t_0,x_0,u})-\phi(\Theta,X_\Theta^{t_0,x_0,u})-w(t_0,x_0)+\phi(t_0,x_0)\right]\\
    =&E\left[w(\Theta,X_\Theta^{t_0,x_0,u})-w(t_0,x_0)+\int_{t_0}^\Theta(f(s,X_s^{t_0,x_0,u},u(s,X_s^{t_0,x_0,u}))-\rho)ds\right]\\
     &-E\left[\phi(\Theta,X_\Theta^{t_0,x_0,u})-\phi(t_0,x_0)+\int_{t_0}^\Theta(f(s,X_s^{t_0,x_0,u},u(s,X_s^{t_0,x_0,u}))-\rho)ds\right]\\
 \leq&-E\left[\phi(\Theta,X_\Theta^{t_0,x_0,u})-\phi(t_0,x_0)+\int_{t_0}^\Theta(f(s,X_s^{t_0,x_0,u},u(s,X_s^{t_0,x_0,u}))-\rho)ds\right],
\end{align*}where $\Theta=\Theta^{t_0,x_0,u}\wedge(t_0+\delta)$.
Hence, we have for any $u(\cdot)\in\mathcal{U}_{t_0}$
\begin{align*}
E\left[\phi(\Theta,X_\Theta^{t_0,x_0,u})-\phi(t_0,x_0)+\int_{t_0}^\Theta [f(s,X_s^{t_0,x_0,u},u(s,X_s^{t_0,x_0,u}))-\rho] ds\right]\leq 0.
\end{align*}
Applying It\^o's formula to $\phi$, dividing $\delta$ for each term, taking the limit of $\delta\to0$ and taking the supremum over $u\in U$, we have
\begin{align*}
  \partial_t\phi(t_0,x_0)+\sup_{u\in U}H(t_0,x_0,u,D\phi,D^2\phi)\leq\rho.
\end{align*}

\emph{Viscosity subsolution property.} For any $\phi\in C^{1,2}([0,\tau]\times\mathbb{R}^d)$, suppose that $w(t,x)-\phi(s,x)$ attains a local maximum at $(t_0,x_0)\in[0,\tau)\times\mathbb{R}^d$ in its neighbourhood $B(t_0,x_0)\subset[0,\tau)\times\mathbb{R}^d$. Let $\Theta^{t_0,x_0,u^h}=\inf\{t\geq t_0,X_t^{t_0,x_0,u^h}\notin B_\varepsilon(x_0)\}$ for some $h>0$, and $\Theta_h=\Theta^{t_0,x_0,u^h}\wedge(t_0+\delta)$, where $u^h(\cdot)\in\mathcal{U}_{t_0}$ is chosen such that
\begin{align*}
  w(t_0,x_0)<E\left[w(\Theta_h,X_{\Theta_h}^{t_0,x_0,u^h})+\int_{t_0}^{\Theta_h}(f(s,X_s^{t_0,x_0,u^h},u^h(s,X_s^{t_0,x_0,u^h}))-\rho)ds\right]+h\delta.
\end{align*}
Then, we have 
\begin{align*}
0\geq& E\left[w(\Theta_h,X_{\Theta_h}^{t_0,x_0,u^h})-\phi(\Theta_h,X_{\Theta_h}^{t_0,x_0,u^h})-w(t_0,x_0)+\phi(t_0,x_0)\right]\\
=& E\left[w(\Theta_h,X_{\Theta_h}^{t_0,x_0,u^h})-w(t_0,x_0)+\int_{t_0}^{\Theta_h}(f(s,X_s^{t_0,x_0,u^h},u^h(s,X_s^{t_0,x_0,u^h}))-\rho)ds\right]+h\delta\\
& -E\left[\phi(\Theta_h,X_{\Theta_h}^{t_0,x_0,u^h})-\phi(t_0,x_0)+\int_{t_0}^{\Theta_h}(f(s,X_s^{t_0,x_0,u^h},u^h(s,X_s^{t_0,x_0,u^h}))-\rho)ds\right]-h\delta\\
>& -E\left[\phi(\Theta_h,X_{\Theta_h}^{t_0,x_0,u^h})-\phi(t_0,x_0)+\int_{t_0}^{\Theta_h}(f(s,X_s^{t_0,x_0,u^h},u^h(s,X_s^{t_0,x_0,u^h}))-\rho)ds\right]-h\delta.
\end{align*}
Applying It\^o's formula to $\phi$, dividing by $\delta$ for each term, taking the supremum over $u\in U$ and taking the limit of $\delta\to0$, we have
\begin{align*}
\partial_t\phi(t_0,x_0)+\sup_{u\in U}H(t_0,x_0,u,D\phi,D^2\phi)\geq\rho-h.
\end{align*}
That completes the proof, since $h>0$ is arbitrary.
\end{proof}

This infinite horizon version of the HJB equations is different from classical HJB equations arsing from finite horizon control problems, where boundary conditions at the terminal time $T$ are given. It is worth mentioning that the terminal values can be seen as forgotten in the process of taking the long-term limit due to the ergodicity of the infinite-horizon control problem. Moreover, the ergodic control value function is periodic in time, that is, $w(t,x)=w(t+\tau,x)$. However, the periodic condition does not make the solution to the HJB equation unique. In fact, it is obvious that if $(w,\rho)$ is a viscosity solution of the HJB equation (\ref{PDE}), then for any $c\in\mathbb{R}$, the pair $(w+c,\rho)$ is also a viscosity solution of (\ref{PDE}). Therefore, the solution could only be unique up to a constant shift for $w$. It is worth studying more properties of the solution of the HJB equation (\ref{PDE}), e.g. uniqueness under the above sense is a very interesting problem. But these questions are beyond the main aim of this article, so we leave them for further study.

\begin{theorem}\label{verification theorem} (Verification Theorem)
Assume that a $\tau$-periodic function $w\in C^{1,2}([0,\tau)\times\mathbb{R}^d)$, which has at most linear growth in $x$, is a classical solution of the HJB equation (\ref{PDE}), and there exists a $\tau$-periodic deterministic measurable function $\bar u(s,x)\in\mathcal{A}$ such that 
\begin{align}\label{opti-veri}
\sup_{u\in U}H(t,x,u,Dw,D^2w)=H(t,x,\bar u,Dw,D^2w).
\end{align}
Then, the corresponding process $\bar u(\cdot)$ is an optimal control such that $\rho=\rho^{\bar\alpha}$. 
Furthermore, if $w(s,x)$ is the value function, we have 
\begin{align*}
    w(t,x)=\lim_{n\to\infty}E\int_t^{n\tau}[f(s,X_s^{t,x,\bar u},\bar u(s,X_s^{t,x,\bar u}))-\rho]ds.
\end{align*}
\end{theorem}
\begin{proof}
Assume $w\in C^{1,2}([0,\tau)\times\mathbb{R}^d)$, then applying It\^o's formula on $w$ from $t$ to $t+\delta$ with $0\leq t<t+\delta<\tau$,
\begin{align*}
Ew&(t+\delta,X_{t+\delta}^{t,x,\bar u})=w(t,x)+E\int_t^{t+\delta}\partial_sw(s,X_s^{t,x,\bar u})ds\\
&+E\int_t^{t+\delta} Dw(s,X_s^{t,x,\bar u})\cdot b(s,X_s^{t,x,\bar u},\bar u(s,X_s^{t,x,\bar u}))ds\\
&+\frac{1}{2}E\int_t^{t+\delta}tr[\sigma\sigma'(s,X_s^{t,x,\bar u},\bar u(s,X_s^{t,x,\bar u}))D^2w(s,X_s^{t,x,\bar u})]ds\\
=&E\int_t^{t+\delta}[\partial_sw(s,X_s^{t,x,\bar u})+H(s,X_s^{t,x,\bar u},\bar u,Dw(s,X_s^{t,x,\bar u}),D^2w(s,X_s^{t,x,\bar u}))-\rho]ds\\
&+E\int_t^{t+\delta}[\rho-f(s,X_s^{t,x,\bar u},\bar u(s,X_s^{t,x,\bar u}))] ds+w(t,x)\\
=&E\int_t^{t+\delta}[\rho-f(s,X_s^{t,x,\bar u},\bar u(s,X_s^{t,x,\bar u}))] ds+w(t,x).
\end{align*}
So, by iteration we have 
\begin{align*}
w(t,x)=&E\left[\int_t^\tau[f(s,X_s^{t,x,\bar u},\bar u(s,X_s^{t,x,\bar u}))-\rho]ds+w(\tau-,X_\tau^{t,x,\bar u})\right]\\
\leq&E\left[\int_t^\tau[f(s,X_s^{t,x,\bar u},\bar u(s,X_s^{t,x,\bar u}))-\rho]ds+w(\tau,X_\tau^{t,x,\bar u})\right]\\
\leq&E\left[\int_t^{n\tau}[f(s,X_s^{t,x,\bar u},\bar u(s,X_s^{t,x,\bar u}))-\rho]ds+w(n\tau,X_\tau^{t,x,\bar u})\right].
\end{align*}
Then, for any $T>0$ that
\begin{align*}
    w(t,x)\leq E\left[\int_t^{T+t}[f(s,X_s^{t,x,\bar u},\bar u(s,X_s^{t,x,\bar u}))-\rho]ds+w(T+t,X_{T+t}^{t,x,\bar u})\right].
\end{align*}
Since the linear growth of $w(t,x)$, we have 
\begin{align*}
\rho\leq \lim_{T\to\infty}\frac{1}{T}E\int_t^{T+t}f(s,X_s^{t,x,\bar u},\bar u(s,X_s^{t,x,\bar u}))ds,
\end{align*}
which means $\rho=\rho^{\bar u}.$ Furthermore, if $w$ is the value function, one has
\begin{align*}
    w(t,x)\leq \lim_{n\to\infty}E\left[w(n\tau-,X_{n\tau}^{t,x,\bar u})+\int_t^{n\tau}(f(s,X_s^{t,x,\bar u},\bar u(s,X_s^{t,x,\bar u}))-\rho)ds\right].
\end{align*}
We claim that $E[w(n\tau-,X_{n\tau}^{t,x,\bar u})]\to0$ as $n\to\infty$. In fact, 
\begin{align*}
   E[w(n\tau-,X_{n\tau}^{t,x,\bar u})]\leq& \lim_{m\to\infty}E\int_{n\tau}^{m\tau}(f(s,X_s^{t,x,\bar u},\bar u(s,X_s^{t,x,\bar u}))-\rho)ds\\
   \leq &\bar\alpha^nC(1+|x|)\to0,~as~n\to\infty.
\end{align*}
If 
\begin{align*}
    E[w(n\tau,X_{n\tau}^{t,x,\bar u})]&\geq E[w(n\tau-,X_{n\tau}^{t,x,\bar u})]>\lim_{m\to\infty}E\int_{n\tau}^{m\tau}[f(s,X_s^{t,x,\bar u},\bar u(s,X_s^{t,x,\bar u}))-\rho]ds,
\end{align*}
it means that $\bar\alpha(\cdot)$ is not the optimal control. By the dynamic programming principle (\ref{DPP}), we have for some small $\delta>0$
\begin{align*}
    E[w(n\tau,X_{n\tau}^{t,x,\bar u})]>E\left[\int_{n\tau}^{n\tau+\delta}[f(s,X_s^{t,x,\bar u},\bar u(s,X_s^{t,x,\bar u}))-\rho]ds+w(n\tau+\delta,X_{n\tau+\delta}^{t,x,\bar u})\right],
\end{align*}
which is a contradiction to the condition (\ref{opti-veri}), since $w\in C^{1,2}$.

As a result, $\bar\alpha$ is an optimal control such that 
\begin{align*}
w(t,x)=\lim_{n\to\infty}E\int_t^{n\tau}[f(s,X_s^{t,x,\bar u},\bar u(s,\mu_s^{t,x,\bar u}))-\rho]ds.
\end{align*}
That obtains the desired results.
\end{proof}

It is worth emphasizing that the control processes we consider in this article are feedback ones. It is not clear that from the HJB equation (\ref{PDE}) one can deduce that the optimal control processes are Lipschitz continuous functions, even if the function $w\in C^{1,2}([0,\tau)\times\mathbb{R}^d)$. The existence of such a function is an interesting problem to consider in a future work.

\section{Connection with controlled ergodic BSDEs} 
For each fixed $u\in\mathcal{U}_t$, we consider the following SDE,
\begin{align*}
\begin{cases}
    dX_s=b(s,X_s,u(s,X_s))ds+\sigma(s,X_s,u(s,X_s))dB_s,\\
    X_t=x,
\end{cases}
\end{align*}
and construct an auxiliary function,
\begin{align*}
w_u(t,x)=\lim_{n\to\infty}\int_t^{n\tau} \left(E[f(s,X_s^{t,x,u},u(s,X_s^{t,x,u}))]-\rho^u\right)ds.
\end{align*}
It is not difficult to find that the function $w_u(x)$ satisfies that for any $T\geq t$
\begin{align*} 
    w_u(t,x)=E\left[\int_t^T(f(s,X_s^{t,x,u},u(s,X_t^{t,x,u}))-\rho^u)ds+w_u(T,X_T^{t,x,u})\right],
\end{align*}
and the controlled ergodic BSDE,
\begin{align*}
-dY_s^{t,x,u}=(f(s,X_s^{t,x,u},u(s,X_s^{t,x,u}))-\rho^u)ds-Z_sdB_s,~~s\geq t,
\end{align*}
has a solution $(Y_s^{t,x,u},Z_s^{t,x,u},\rho^u)$ with $Y_s^{t,x,u}=w_u(s,X_s^{t,x,u})$ and $\rho^u$ as above. Moreover, we have solved the recursive optimal control problem with the value function
\begin{align*}
w(t,x)=\underset{u\in\mathcal{U}_t}{esssup}~ Y_t^{t,x,u}.
\end{align*}
Here we have seen from our results in Sections \ref{dynamic programming principle} and \ref{HJB} that $w(t,x)$ satisfies the related dynamic programming equation: for any $t\leq T<\tau$
\begin{align*}
w(t,x)=\sup_{u\in\mathcal{U}_t}E\left[\int_t^T(f(s,X_s^{t,x,u},u(s,X_s^{t,x,u}))-\rho) ds+w(T,X_T^{t,x,u})\right].
\end{align*}
and is a viscosity solution of the related HJB equation,
\begin{align*}
\partial_t w(t,x)+\sup_{u\in U}\{f(t,x,u)+D\varphi(t,x)\cdot b(t,x,u)+\frac{1}{2}tr(\sigma\sigma'(t,x,u)D^2\varphi(t,x))\}=\rho.
\end{align*}

Let us single out one important special case. If the diffusion of SDE (\ref{dynamics}) does not contain the control variable, \emph{i.e.}, $\sigma(s,x,u)\equiv\sigma(s,x)$ for all $s\geq t$ and $x\in\mathbb{R}^d$, $b,\sigma,f$ satisfy Assumption \ref{assume}, then the HJB equation (\ref{PDE}) should be 
\begin{align}\label{new pde}
  \partial_t w(t,x)+\frac{1}{2}tr(\sigma\sigma'(t,x)D^2\varphi(t,x))+\sup_{u\in U}\{f(t,x,u)+D\varphi(t,x)\cdot b(t,x,u)\}=\rho.
\end{align}  
Define $h(s,x,z)=\sup_{u\in U}\{f(s,x,u)+z\cdot b(s,x,u)/\sigma(s,x)\}$ and consider a non-controlled SDE 
\begin{align*}
  \begin{cases}
   dX_s=\sigma(s,X_s)dB_s,~~~s\geq t,\\
   X_t=x.
  \end{cases}
\end{align*}
Then the following infinite-horizon BSDE
\begin{align*}
-dY_s=(h(s,X_s,Z_s)-\rho)ds-Z_sdB_s,~~s\geq t,
\end{align*}
has a solution $(Y_s^{t,x},Z_s^{t,x},\rho)$ and $w(s,x)=Y_s^x$ is a viscosity solution of the PDE (\ref{new pde}). Moreover, $w(s,X_s^{t,x})=Y_s^{t,x}$, a.s., for all $s\geq t$.

\section*{Conclusions}
In this article, we have mainly studied the average stochastic optimal control problems. We have proved the convergence results of the payoff functional, which is ergodic, and then we have constructed an auxiliary function $w(t,x)$, based on which we have established the dynamic programming principle and given the HJB equation. Finally, we have proved that the pair $(w,\rho)$ is a viscosity solution of the HJB equation. The construction of the auxiliary function $w(t,x)$ can inspire the proof of the maximum principle of the average stochastic optimal control problems, which is worth pursuing in the future.

\bibliographystyle{plain}
\bibliography{reference}

\end{document}